\documentclass[12pt,a4paper, reqno]{amsart}
\usepackage{amsfonts,amsmath,amsthm,amssymb, amscd}
\usepackage{enumerate}

\newcommand*{\gp}[1]{\langle\;#1\;\rangle}
\newcommand*{\order}[1]{\vert #1 \vert}

\newtheorem{theorem}{Theorem} 

\newtheorem{lemma}{Lemma}
\newtheorem{corollary}{Corollary}

 \DeclareMathOperator{\auto}{Aut}

\begin{document}
\title[Unitary Subgroups of commutative group algebras]{Unitary Subgroups of commutative group algebras of characteristic two}

\author{Zsolt Balogh, Vasyl Laver}

\address{
Department of Mathematical Science, United Arab Emirates University, UAE;
Department of Informative and Operating Systems and Technologies, Uzhhorod National University, Ukraine}
\email{\{baloghzsa,v.laver\}@gmail.com}

\thanks{The research was supported by UAEU Research Start-up (1) No. 31S370}


\begin{abstract}
Let $FG$ be the group algebra of a finite $2$-group $G$ over a finite field $F$ of characteristic two and $\circledast$ an involution which arises from $G$. The $\circledast$-unitary subgroup of $FG$, denoted by $V_{\circledast}(FG)$, is defined to be the set of all normalized units $u$ satisfying the property $u^{\circledast}=u^{-1}$. In this paper we establish the order of $V_{\circledast}(FG)$ for all involutions $\circledast$ which arise from $G$, where $G$ is a finite cyclic $2$-group and show that all $\circledast$-unitary subgroups of $FG$ are not isomorphic.
\end{abstract}

\maketitle

\section{Introduction}

Let $FG$ be the group algebra of a $2$-group $G$ over a finite field $F$ of characteristic $2$. The set of all units in $FG$ with augmentation $1$ forms a group. This group, denoted by $V(FG)$, is called group of normalized units. The description of the structure of $V(FG)$ is a central problem in the theory of group algebras and it has been investigated by several authors. 
For an excellent survey on the group of units of the modular group algebras we refer the reader to \cite{Bovdi_survey}.

Let $\circledast$ be an involution of $FG$. An element $u \in V(FG)$ is called $\circledast$-unitary if $u^{\circledast}=u^{-1}$. The set of all $\circledast$-unitary elements of $V(FG)$ forms a subgroup of $V(FG)$ which is denoted by $V_{\circledast}(FG)$. The unitary subgroup related to the canonical involution (the $F$-linear extension of the involution on $G$ which sends each element of $G$ to its inverse) plays an important role of studying the structure of the group of units of group algebras \cite{Bovdi_Unitarity, Bovdi_Szakacs_II, Bovdi_Kovacs_I, Spinelli_2, Spinelli_1, Novikov}.

To establish the order of $V_*(FG)$ is a particularly challenging problem if the characteristic of $F$ is two. This subject has been investigated in several papers. A. Bovdi and Szak\'acs determined the structure of $*$-unitary subgroups for all abelian $p$-groups and finite fields of characteristic $p$ in \cite{Bovdi_Szakacs_III} and \cite{Bovdi_Szakacs_II}.  In \cite{Bovdi_Grichkov} V. Bovdi and Grichkov determined the invariants of the $\eta$-unitary and symmetric normalized units of $FG$, where $F$ is a field of two elements, $G$ is a finite abelian $2$-group, and $\eta$ is an involutory involution.

We know only partial results when $G$ is a non-abelian group. In \cite{Bovdi_Rosa_I} V. Bovdi and Rosa determined the order of $V_*(FG)$ for dihedral, quaternion and extraspecial $2$-groups if $F$ is a finite field of characteristic $2$. The structure of $V_*(FG)$  when $F$ is the field of two elements and $G$ is a group of order $16$ or a group of maximal class was established in \cite{Bovdi_Erdei_II} and \cite{Bovdi_Erdei_I}, respectively.
In \cite{Bovdi_Kovacs_I} the authors described all group algebras whose $*$-unitary subgroup is normal in $V(FG)$. 
The structures of $V_*(FQ_8)$ and $V_*(FD_8)$ were established in \cite{Creedon_Gildea_I} and \cite{Creedon_Gildea_II} where $Q_8$ is the quaternion, $D_8$ is the dihedral group of order $8$, respectively and $F$ is a finite field of characteristic two. 

There are only limited number of results in non-modular case. In \cite{Sharma_I} the order of $V_*(F_{2^k}D_{2N})$ was determined, where $D_{2N}$ is the dihedral group of order $2N$. 

A. Bovdi and Szak\'acs also gave the order of the $\circledast$-unitary subgroups when the characteristic of field $F$ is odd and $\circledast$ arises from an abelian $p$-group $G$ in \cite{Bovdi_Szakacs_I}. Additionally, in \cite{Bovdi_Rozgonyi} the structure of unitary subgroups have been investigated for different involutions, where $G$ is the dihedral group. In \cite{Balogh_Creedon_Gildea} the structure of  
$V_{\circledast}(FG)$ was described for all non-abelian groups $G$, where the order of $G$ is $8$ and $\circledast$ arises from $G$.

In this paper we establish the order of $V_{\circledast}(FG)$, where $\circledast$ is an involution which arises from a finite cyclic $2$-group $G$. Using these results we prove the following theorem. 

\begin{theorem}\label{theorem_1}
	Let $F$ be a finite field of characteristic two. Then the $\circledast$-unitary subgroups of $FC_{2^n}$ $(n>2)$, where $\circledast$ arises from $C_{2^n}$ are not isomorphic.
\end{theorem}

\section{$\circledast$-unitary subgroups of $FC_{2^n}$}

Let $G$ be a finite $2$-group. We will denote by $G[2^i]$ the subgroup of $G$ generated by the elements of order $2^i$.
We use the notation $G^{2^i}$ for the subgroup $\gp{ g^{2^i} \,\vert\,g \in G }$. Throughout this paper $|S|$ denotes the cardinality of the finite set $S$, $\order{g}$ -- the order of $g\in G$, $C_n$ -- cyclic group of order $n$, and $\auto{G}$ denotes the automorphism group of the group $G$. Also, in the sequel of the paper we consider $F$ as finite field of characteristic two, and where it is needed we specify the order of the field in the lower index (i.e. $F_{2^n}$ is a field of $2^n$ elements).

The following two lemmas will be useful.
\begin{lemma}(\cite[Theorem 2]{Bovdi_Szakacs_III})\label{szakacs}
	Let $G$ be a finite abelian $2$-group and ${F}$ a finite field of characteristic two.
	Then \[\order{V_*({F}G)}=\order{G^2[2]}\cdot |F|^{\frac{1}{2}(\order{G}+\order{G[2]})-1}.\]
\end{lemma}

\begin{lemma}(\cite[Proposition 16, p.135]{Dummit_book})\label{dummit}
	$\auto{C_n} \cong (\mathbb{Z}_n)^\times$.
\end{lemma}
From Lemma \ref{dummit} and \cite[Theorem 2', p.43]{Book_Ireland_Rosen} we get 

\begin{corollary}\label{automorphism_C}
	$\auto{C_{2^n}} \cong C_{2^{n-2}}\times C_2$ $(n>2)$.
\end{corollary}

In the remainder of the paper we assume $n$ to be greater than $2$.

Therefore the elements of order two in $\auto{C_{2^n}}$ from a subgroup which is isomorphic to the Klein group. 
Let $\sigma_1$ be the identity
automorphism and
\[
\sigma_2 : 
a \mapsto a^{-1};
\qquad
\sigma_3 : 
a \mapsto a^{2^{n-1}-1} ;
\qquad
\sigma_4 :
a \mapsto a^{2^{n-1}+1}. 
\]

The involution $\sigma_2$ is the canonical involution of $FG$. According to Lemma \ref{szakacs} we have that $\order{V_*(FC_{2^n})}=2\cdot \order{F}^{\frac{\order{C_{2^n}}}{2}}$.

We will denote by $\circledast$ the linear extension of the $\circledast$ involution of $G$ into $FG$ and the set of $\circledast$-symmetric elements of $G$ by $G_{\circledast}=\{ g\in G \;|\; g=g^{\circledast} \}$. Every $\circledast$-symmetric element of $FG$ (i.e. $x=x^{\circledast}$) can be written in the form
\[
\sum_{g\in G_{\circledast}} \alpha_g g + \sum_{g\not\in G_{\circledast}} \beta_g (g+g^{\circledast }).
\] To avoid confusion in the sequel of the paper, where it is needed, we will specify the involution by writing $\sigma_i$, $i\in \{3,4\}$ instead of $\circledast$.

To give a formula for the order of the unitary subgroup with respect to $\sigma_3$ we need more consideration.

Let $H$ be a normal subgroup of $G$ and $I(H)$ the ideal of $FG$ generated by the set $\{(1+h) \;|\; h\in H\}$. $I(H)$ can be considered as an $F$-modulus with basis elements $u(1+h)$, where $u \in R(G/H)$ and $h\ne 1$. 
As a conclusion we have $|I(H)|=F^{\frac{|C_{2^n|}}{2}}$. 
It is well-known that $FG/I(H)\cong F(G/H)$ and we denote by $\Psi$ the corresponding natural homomorphism. Let us denote by  $\overline{G}=G/H$ and by $V_{\sigma_3}(F\overline{G})$ the unitary subgroup of the factor algebra $FG/I(H)$, where $\overline{x}^{\sigma_3}$ is the induced action of the involution $\sigma_3$ on $\overline{x} \in FG/I(H)$. It is clear that the set 
\[ N^{\sigma_3}_{\Psi}=\{x\in V(FG) \;|\; \Psi(x) \in V_{\sigma_3}(F\overline{G}) \} \]
forms a subgroup in $V(FG)$. Furthermore, the set $I(H)^+=\{1+x \, |\, x\in I(H) \}$ forms a normal subgroup in $V(FG)$. We define $S_H$ to be the group generated by the elements $\{ xx^{\sigma_3} \;|\; x\in N^{\sigma_3}_{\Psi} \}$. Based on the fact that $xx^{\sigma_3} \in 1+\ker(\Psi)=I(H)^+$
we can see that $S_H$ is a subgroup of $I(H)^+$.

\begin{lemma}\label{lemma_3}
	Let $F$ be a finite field of characteristic two. Then the order of
	$V_{\sigma_3}(FC_{2^n})$ equals $\order{F}^{2^{n-1}}$. 
\end{lemma}

\begin{proof}
Let $H=G_{\sigma_3}=\gp{a^{2^{n-1}}}$ be. Let us denote by $\widehat{H}$ the sum of all the elements of $H$. 
It can be seen that $xx^{\sigma_3}\in I(H)^+$ for all $x\in N^{\sigma_3}_{\Psi}$. Therefore 

\begin{equation}
xx^{\sigma_3}=\sum_{g\in G_{\sigma_3}} \delta_g g + \sum_{g\not\in G_{\sigma_3}} \alpha_g (g+g^{\sigma_3})\widehat{H}+\beta a^{2^{n-2}}\widehat{H}
\end{equation}
for some $\alpha_g,\beta,\delta_g \in F$.

We will prove that $S_H$ is generated by the elements $a^{2^{n-1}}$, $1+\beta a^{2^{n-2}}\widehat{H}$ and $1+\alpha_i (a^i+a^{2^{n-1}-i}) \widehat{H}$,  where $i\in\{1,\cdots,2^{n-2}-1\}$. 

Let $x_i=1+\alpha_i (a^i+a^{{2^{n-1}}+i})$ be. Then 
\[
\begin{split}
x_ix_i^{\sigma_3}&=(1+\alpha_i a^i+\alpha_i a^{{2^{n-1}}+i})(1+\alpha_i a^{{2^{n-1}}-i}+\alpha_ia^{-i})\\&=1+\alpha_i (a^i+a^{2^{n-1}-i}+a^{2^{n-1}+i}+a^{-i})=1+\alpha_i (a^i+a^{2^{n-1}-i}) \widehat{H}.
\end{split}
\]
Since $x_ix_i^{\sigma_3}\in I(H)^+$, $x_i$ belongs to $N^{\sigma_3}_{\Psi}$ so we have proved that $1+\alpha_i (a^i+a^{2^{n-1}-i}) \widehat{H} \in S_H$ for every $\alpha_i \in F$.

Evidently, $y=1+\gamma (a^{2^{n-3}}+a^{-{2^{n-3}}})\ne 1$ is a $\sigma_3$-symmetric element if $n>3$ and $y=1+\gamma (a+a^3)$ is also a $\sigma_3$-symmetric element if $n=3$.
Then we can calculate that
\[
\begin{split}
yy^{\sigma_3}&=y^2=1+\gamma^2 (a^{2^{n-2}}+a^{-{2^{n-2}}})=1+\gamma^2 a^{2^{n-2}}\widehat{H}.
\end{split}
\]

It is well-known that the group of units of $F$, denoted by $U(F)$, is a cyclic group of odd order. Therefore $\eta(\alpha)=\alpha^2$ is an automorphism of $U(F)$ and we can pick $\gamma\in F$ such that $\gamma^2=\beta$. Therefore $1+\beta a^{2^{n-2}}\widehat{H}$ also belongs to $S_H$. 

It is easy to see that $a^{2^{n-1}}$ belongs to $S_H$. Indeed, $a\cdot a^{\sigma_3}=a^{2^{n-1}}$. 

Now, we have to prove that $\gamma a^{2^{n-1}}$ does not belong to $S_H$, for any $\gamma \ne 1$. 
Every element of $FG$ can be expressed as
$x_1+x_2a$, where $x_1,x_2 \in FC_{2^n}^2$. A simple calculation shows that 
\[
(x_1+x_2a)(x_1+x_2a)^{\sigma_3}=x_1x_1^*+x_2x_2^*a^{2^{n-1}}+(x_1^*x_2+x_1x_2^*a^{2^{n-1}})a,
\]
where $*$ is the canonical involution of $FG$. According to \cite[Lemma 3]{Balogh_Bovdi_II}  $a^{2^{n-1}}$ does not belong to the support of $zz^*$ for every $z\in FG$ and the trace of $zz^*$ is equal to the augmentation of $z$. Thus, $x_1x_1^*+x_2x_2^*a^{2^{n-1}}=\gamma a^{2^{n-1}}$ if and only if $\gamma=1$. 

It is easy to check that  
\[
\begin{split}
\big(1+\alpha_i (a^i+a^{2^{n-1}-i}) \widehat{H} \big)&
\big(1+\beta_j (a^j+a^{2^{n-1}-j}) \widehat{H} \big)=\\
&1+\big(\alpha_i(a^i+a^{2^{n-1}-i})+\beta_j (a^{2^{n-1}+j}+a^{-j})\big)\widehat{H}
\end{split}
\]
and
\[
\begin{split}
\big(1+\alpha_i (a^i+a^{2^{n-1}-i}) \widehat{H} \big)&\big(1+\beta a^{2^{n-2}}\widehat{H} \big)=\\
&1+\big(\alpha_i(a^i+a^{2^{n-1}-i})+\beta a^{2^{n-2}}\big)\widehat{H}.
\end{split}
\]
Thus, $S_H$ being elementary abelian group confirms that $\order{S_H}=2\cdot \order{F}^{2^{n-2}}$.

According to the fact that
$|I(H)^+|=|I(H)|$ and the homomorphism theorem we have that 
\[
|V_{\sigma_3}(FC_{2^n})|= \frac{|I(H)|\cdot|V_*(F\overline{G})|}{|S_H|}=\order{F}^{\frac{\order{C_{2^n}}}{2}}\cdot \frac{|V_*(F\overline{G})|}{|S_H|}.
\]

Using Lemma \ref{szakacs} we get
\[
|V_{\sigma_3}(FC_{2^n})|=\frac{\order{F}^{2^{n-1}}\cdot 2\cdot \order{F}^{2^{n-2}}}{2\cdot \order{F}^{2^{n-2}}}=\order{F}^{2^{n-1}},
\]
which proves the lemma.

\end{proof}

Let us consider the case when $\circledast = \sigma_4$. Note that indices of all the coefficients in the statement and the proof of the following Lemma are considered as elements of $F_{2^n}$.


\begin{lemma}\label{lemma_5}
	Let $F$ be a finite field of characteristic two and $x=\displaystyle\sum_{i=0}^{2^n-1}\alpha_i a^i$, $x\in V(FC_{2^n})$.
	Then 
	\[
	\begin{split}
	xx^{\sigma_4}&=\sum_{i=0}^{2^{n-2}-1}  (\alpha_{2i}+\alpha_{2i+2^{n-1}})^2 a^{4i}\\
	&+\sum_{i=0}^{2^{n-2}-1}  (\alpha_{2(i+2^{n-3})+1}+\alpha_{2(i+2^{n-3})+1+2^{n-1}})^2 a^{4i+2}\\
	&+\sum_{j=0}^{2^{n-2}-1} \sum_{i=0}^{2^{n-2}-1} (\alpha_{2i}+\alpha_{2i+2^{n-1} })(\alpha_{2j+1-2i+2^{n-1} }+\alpha_{2j+1-2i})\cdot\\
	&\cdot  (a^{2j+1}+a^{2j+1+2^{n-1} }).
	\end{split}
	\]

\end{lemma}

\begin{proof}
	It is clear that
	$x^{\sigma_4}=\sum_{i=0}^{2^{n-1}-1}\alpha_{2i} a^{2i}+\sum_{i=0}^{2^{n-1}-1}\alpha_{2i+1+2^{n-1}} a^{2i+1}$.
	Thus the coefficient of $a^l$ in $xx^{\sigma_4}$, where $l \in \{1,3,\ldots,2^n-1\}$ is 
	\[
	\sum_{i=0}^{2^{n-1}-1} \alpha_{2i}\alpha_{l-2i+2^{n-1}}+
	\sum_{i=0}^{2^{n-1}-1} \alpha_{2i+1}\alpha_{l-2i-1}.
	\]
	Because $l-2i-1$ is even we can rearrange the second term:
	\[\sum_{i=0}^{2^{n-1}-1} \alpha_{2i+1}\alpha_{l-2i-1}=\sum_{k=0}^{2^{n-1}-1} \alpha_{l-2k}\alpha_{2k},
	\] where $2k = l-(2i+1)$.
	Therefore,
	\[
	\begin{split}
		&\sum_{i=0}^{2^{n-1}-1} \alpha_{2i}\alpha_{l-2i+2^{n-1}}+
		\sum_{i=0}^{2^{n-1}-1} \alpha_{2i+1}\alpha_{l-2i-1}\\&=	\sum_{k=0}^{2^{n-1}-1} \alpha_{2k}(\alpha_{l-2k+2^{n-1}}+\alpha_{l-2k}).
	\end{split}
	\]

Consider the coefficient $\alpha_{2k+2^{n-1}}=\alpha_{2(k+2^{n-2})}$ for some $k$. We have that
\[
\alpha_{l-2(k+2^{n-2})+2^{n-1}}+\alpha_{l-2(k+2^{n-2})}=\alpha_{l-2k}+\alpha_{l-2k+2^{n-1}}.
\]
As a consequence we have proved that the coefficients of $a^l$ and $a^{l+2^{n-1}}$ are equal.
Therefore	
\[
\begin{split}
&\sum_{k=0}^{2^{n-1}-1} \alpha_{2k}(\alpha_{l-2k+2^{n-1}}+\alpha_{l-2k})\\
	&=	\sum_{k=0}^{2^{n-2}-1} (\alpha_{2k}+\alpha_{2k+2^{n-1}})(\alpha_{l-2k+2^{n-1}}+\alpha_{l-2k}).
	\end{split}
\]

	Let us consider the case when $l$ is even. Then the coefficient of $a^l$ in $xx^{\sigma_4}$ is 
	\begin{equation}\label{eq1}
	\sum_{i=0}^{2^{n-1}-1} \alpha_{2i}\alpha_{l-2i}+
	\sum_{i=0}^{2^{n-1}-1} \alpha_{2i+1}\alpha_{l-2i-1+2^{n-1}}.
	\end{equation}
	Assume that $l\equiv 4t \pmod{2^n}$ for some $t \in \{0,\ldots,2^{n-2}\}$. 	Consider the first summand of (\ref{eq1}).	Let $2i \equiv l-2i \pmod{2^n}$. Then, as $l\equiv 4t \pmod{2^n}$, we have that $4i \equiv 4t \pmod{2^n}$ or $2i \equiv 2t \pmod{2^{n-1}}$. Therefore if $2i \equiv 2t \pmod{2^{n-1}}$ then $\alpha_{2i}\alpha_{l-2i}$ is equal either to $\alpha_{2t}^2$ or to $\alpha_{2t+2^{n-1}}^2$. Now, suppose that $2i \not\equiv l-2i \pmod{2^n}$. Consider $s\in\{0,\ldots,2^{n-1}-1\}$ such that $2s \equiv l-2i \pmod{2^{n}}$. Then $l-2s \equiv  l - (l-2i) \equiv 2i \pmod{2^{n}}$, i.e. each of the summands appear twice, so
	\[
	\sum_{i=0}^{2^{n-1}-1} \alpha_{2i}\alpha_{l-2i}=\alpha_{2t}^2+\alpha_{2t+2^{n-1}}^2 = (\alpha_{2t}+\alpha_{2t+2^{n-1}})^2.
	\]
	
	Consider the second summand of (\ref{eq1}). 

	Suppose that there exists $i \in \{0,\ldots,2^n-1\}$ such that $2i+1 \equiv l-(2i+1)+2^{n-1} \pmod{2^{n}}$. Taking into account that $l\equiv 4t \pmod{2^n}$ we get $2(2i+1)\equiv 4t+2^{n-1} \pmod{2^{n}}$. So $2i+1\equiv 2t+2^{n-2} \pmod{2^{n-1}}$ which is a contradiction. This means that $2i+1 \not\equiv l-(2i+1)+2^{n-1} \pmod{2^{n}}$ for each $i \in \{0,\ldots,2^n-1\}$.
	
	Consider $s\in\{0,\ldots,2^{n-1}-1\}$ such that $2s+1 \equiv l-(2i+1)+2^{n-1} \pmod{2^{n}}$. Then 
	\[
	l-(2s+1)+2^{n-1} \equiv  l - (l-(2i+1)+2^{n-1}) + 2^{n-1} \equiv 2i+1.
	\]
	This means that each summand of the form  $\alpha_{2i+1}\alpha_{l-2i-1+2^{n-1}}$ appears twice, thus we can conclude that 
	\[
	\sum_{i=0}^{2^{n-1}-1} \alpha_{2i+1}\alpha_{j-2i-1+2^{n-1}}=0.
	\]

	Now assume that $l\equiv 4t+2 \pmod{2^n}$. Then by similar considerations $\sum_{i=0}^{2^{n-1}-1} \alpha_{2i+1}\alpha_{l-2i-1+2^{n-1}}=\alpha_{2(t+2^{n-3})+1}^2+\alpha_{2(t+2^{n-3})+1+2^{n-1}}^2$ and $\sum_{i=0}^{2^{n-1}-1} \alpha_{2i}\alpha_{l-2i}=0$. 
	
	Finally, summing up the coefficients for all $a^l$, $l\in\{0,\ldots,2^n-1\}$ we have proved the statement of the lemma.
\end{proof}

\begin{corollary}\label{cor_1}
	Let $F$ be a finite field of characteristic two and $x=\displaystyle\sum_{i=0}^{2^n-1}\alpha_i a^i\in V(FC_{2^n})$.
	Then 
	\[
	xx^{\sigma_4}=\sum_{i=0}^{2^{n-1}-1} \beta_{2i} a^{2i}+\sum_{i=0}^{2^{n-1}-1} \beta_{2i+1} a^{2i+1},
	\]
	and $\beta_{2i+1}=f_{i}(\beta_0,\beta_2,\ldots,\beta_{2^n-2})$ for some functions $f_{i}$, where ${0 \leq i < 2^{n-1}}$. Furthermore, $\sum_{i=0}^{2^{n-1}-1} \beta_{2i+1}=0$. 
\end{corollary}
\begin{proof}
	Taking into account Lemma \ref{lemma_5} and that $\eta(\delta)=\delta^2$ is an automorphism on $F$ we can conclude that the coefficient $\beta_k$ depends only on the coefficients $\beta_0,\beta_2,\ldots,\beta_{2^n-2}$ for every odd $k$.
	
	According to Lemma \ref{lemma_5}, the coefficients of $a^{2j+1}$ and $a^{2j+1+2^{n-1}}$ coincide for each $0\leq j <2^{n-1}$, therefore $\sum_{i=0}^{2^{n-1}-1} \beta_{2i+1}=0$.
\end{proof}

\begin{lemma}\label{lemma_6}
	Let $F$ be a finite field of characteristic two. Then
	$V_{\sigma_4}(FC_{2^n})$ is an elementary abelian group of order $\order{F}^{2^{n-1}}$.
\end{lemma}

\begin{proof}
	The mapping $\varphi(x)=xx^{\sigma_4}$ is a homomorphism on $V(FG)$. Denote 
	by $S_{\sigma_4}$ the image of this homomorphism.
	The kernel of $\varphi(x)$ coincides with the unitary subgroup $V_{\sigma_4}(FG)$.
	Therefore 
	\[
	V(FG)/V_{\sigma_4}(FG) \cong S_{\sigma_4}.
	\]
	
	According to Corollary \ref{cor_1}, the number of the free coefficients in $xx^{\sigma_4}$ is $2^{n-1}-1$. Indeed, $\beta_k$'s are not independent variables for odd $k$ and $xx^{\sigma_4}$ being unit implies $\sum_{i=0}^{2^{n}-1} \beta_{i}=1$. We see that	 $\order{S_{\sigma_4}}=\order{F}^{2^{n-1}-1}$.
	Therefore 
	\[
	\order{V_{\sigma_4}(FG)}=\frac{\order{V(FG)}}{\order{S_{\sigma_4}}}=\frac{\order{F}^{2^n-1}}{\order{F}^{2^{n-1}-1}}=\order{F}^{2^{n-1}}.
	\]
	
	Now, we will prove that $V_{\sigma_4}(FC_{2^n})$ is elementary abelian. According to Lemma \ref{lemma_5} if $xx^{\sigma_4}=1$, then $x$ is a $\sigma_4$-symmetric element. Therefore $x^2=1$ which proves the lemma.
	
\end{proof}

\begin{proof}[Proof of the main theorem]
	Since the order of the $*$-unitary subgroup is the largest it cannot be isomorphic to any $\circledast$-unitary subgroup.
	
	Since $(a^{2i})^{\sigma_3}=a^{-2i}$ and $(a^{2i+1})^{\sigma_3}=a^{2^{n-1}-2i-1}$ we conclude that $V_{\sigma_3}(FC_{2^n})\cap C_{2^n}=\gp{a^2}$. Therefore the exponent of $V_{\sigma_3}(FC_{2^n})$ is not less than $2^{n-1}$. Keeping in mind that  $V_{\sigma_4}(FC_{2^n})$ is elementary abelian by Lemma \ref{lemma_5}  we conclude that the unitary subgroups related to the involutions $\sigma_3$ and $\sigma_4$ are not isomorphic groups.
	
\end{proof}

However, in the general case when $G$ is abelian Theorem \ref{theorem_1} does not necessarily hold. Using GAP System \cite{GAP4} we can verify that some unitary subgroups of $F(C_8 \times C_2)$ , where $F$ is a field of two elements, are isomorphic. In this case there are six automorphisms of order $\leq 2$:
\[
\sigma_1 = \begin{cases}
a &\mapsto a\\
b &\mapsto b 
\end{cases};
\quad
\sigma_2 = \begin{cases}
a &\mapsto a^3 b\\
b &\mapsto b
\end{cases};
 \quad 
 \sigma_3 = \begin{cases}
a &\mapsto a b\\
b &\mapsto b
\end{cases};
\]
\[
\sigma_4 = \begin{cases}
a &\mapsto a\\
b &\mapsto a^2 b
\end{cases}; \quad
\sigma_5 = \begin{cases}
a &\mapsto a^3\\
b &\mapsto a^2 b
\end{cases}; \quad
\sigma_6 = \begin{cases}
a &\mapsto a^3\\
b &\mapsto b
\end{cases}.
\]
Unitary subgroups, induced by $\sigma_2$ and $\sigma_4$ are isomorphic to Klein four-group $C_2 \times C_2$.


\end{document}